\documentclass[10pt, reqno]{amsart}
\usepackage{mathscinet}
\usepackage[utf8]{inputenc}
\usepackage[english]{babel}

\usepackage[square,sort,comma,numbers]{natbib}

\usepackage[dvipsnames]{xcolor}
\usepackage{float}

\usepackage{graphicx}
\usepackage{tikz}
\usetikzlibrary{arrows}
\usetikzlibrary{matrix,arrows,decorations.pathmorphing}

\usepackage{amsmath}
\usepackage{amssymb}
\usepackage{amsthm}

\usepackage{stmaryrd}
\usepackage[all]{xy}
\usepackage[mathcal]{eucal}
\usepackage{verbatim}\usepackage{fullpage} \usepackage{wrapfig}
\usepackage{lipsum}
\usepackage[all]{xy}
\theoremstyle{plain}
\newtheorem{thm}{Theorem}[section]

\newtheorem{lem}[thm]{Lemma}

{\leftskip=1cm}

\newtheorem{maintheorem}{Theorem}

\theoremstyle{definition}

\newtheorem{exmp}{Example}

\newcommand{\nc}{\newcommand}
\nc{\dmo}{\DeclareMathOperator}
\DeclareMathOperator{\PConf}{PConf}

\DeclareMathOperator{\Conf}{Conf}

\DeclareMathOperator{\Ho}{H}
\DeclareMathOperator{\DD}{\mathbb{B}}
\DeclareMathOperator{\BB}{\mathbb{U}}
\DeclareMathOperator{\RR}{\mathbb{R}}

\newcommand{\ZZ}{\mathbb{Z}}

\nc{\para}[1]{\medskip\noindent\textbf{#1.}}

\title{Adding a point to configurations in closed balls}
\author{Lei Chen, Nir Gadish, and Justin Lanier}

\begin{document}

	\begin{abstract}
		We answer the question of when a new point can be added in a continuous way to configurations of $n$ distinct points in a closed ball of arbitrary dimension. We show that this is possible given an ordered configuration of $n$ points if and only if $n \neq 1$. On the other hand, when the points are not ordered and the dimension of the ball is at least 2, a point can be added continuously if and only if $n = 2$. These results generalize the Brouwer fixed-point theorem, which gives the negative answer when $n=1$. We also show that when $n=2$, there is a unique solution to both the ordered and unordered versions of the problem up to homotopy.
	\end{abstract}

	\maketitle
	
	\section{Introduction}
	
	Let $\DD^m$ be the closed ball of dimension $m$, with $m \geq 1$. This paper answers the following basic question:
	\begin{quotation}

		\emph{Is there a continuous rule that adds a new distinct point to every configuration of $n$ distinct points in $\DD^m$?}
	\end{quotation}
	The challenge here is that the new point must be distinct from all of the existing ones, and it is not clear whether such a choice can be made continuously.
	
	\begin{exmp}[\textbf{Case $n=1$: Brouwer fixed-point theorem}]\label{ex:case1}
		Given one point in $\DD^m$, a continuous choice of a second distinct point can be thought of as a continuous function from the closed ball to itself with no fixed points. By the Brouwer fixed-point theorem \cite{Brouwer} no such continuous function exists, and therefore introducing a second distinct point continuously is impossible.
	\end{exmp}
	When extending the Brouwer fixed-point theorem to $n>1$, the question splits into two versions: either the $n$ points are given with an ordering $(p_1, \ldots, p_n)$, or the points instead form an unordered set. The first author addressed both versions of this question with respect to point configurations lying in the infinite plane, in the $2$-sphere, and in closed surfaces $S_g$ with $g \geq 2$ in \cite{Chen} and in her joint work with Salter \cite{CS}.
	
	It is perhaps surprising that having more points to avoid does not necessarily make the task of adding a point harder.
	
	\begin{exmp}[\textbf{Case $n=2$: midpoint}]\label{ex:case2}
		Given two distinct points in $\DD^m$, their midpoint provides a continuous way to introduce a third point, distinct from the existing two. In fact, in \S\ref{sec:n=2} we show that this midpoint construction is the unique way to add a point when $n=2$ up to homotopy.
	\end{exmp}
	
	Let us state more formally the problem of continuously adding a point to a configuration of $n$ points in $\DD^m$, which from here on will be abbreviated to $\DD$. Let $\PConf_{n}(\DD)$ denote the \emph{pure configuration space} of $n$ distinct ordered points in $\DD$, topologized as a subspace of $(\DD)^n$. A continuous map $\PConf_{n}(\DD)\to \DD$ is said to `add a point' if the image of every configuration $(p_1,\ldots,p_n)$ is a point $p_0$ distinct from all $p_i$ with $i\geq 1$. Equivalently, consider the map $g_{m,n}:\PConf_{n+1}(\DD)\to \PConf_{n}(\DD)$ that forgets the $0$th point $(p_0,p_1,\ldots,p_n)\mapsto (p_1,\ldots,p_n)$. The problem of continuously introducing a new point is precisely the question of finding a continuous section for $g_{m,n}$.
	
	Next, the symmetric group $\Sigma_n$ acts on ordered configurations by permuting the indices, and the quotient $\Conf_n(\DD):=\PConf_{n}(\DD)/\Sigma_n$ is the \emph{configuration space} of $n$ distinct unordered points. The map $g_{m,n}$ forgetting the point $p_0$ is $\Sigma_n$-equivariant and descends to a map $\bar{g}_{m,n}: \PConf_{n+1}(\DD)/\Sigma_n \to \Conf_{n}(\DD)$. With this, our problem of adding a point distinct from a given unordered set of $n$ points is the problem of finding a section for $\bar{g}_{m,n}$.
	
	The easy positive result for $n=2$ in Example \ref{ex:case2} suggests that the problem for larger $n$ might also have a solution. However, this is where the two versions of the problem part ways: the ordered version extends the $n=2$ case by admitting similar solutions, while the unordered version reverts back to the behavior at $n=1$.
	
	\begin{maintheorem}[\textbf{Ordered}]\label{theorem:yessection}
		In dimension $m \geq 1$ and with $n \geq 1$ points, the forgetful map $g_{m,n}$ has a section if and only if $n \neq 1$. That is, one \underline{can} continuously add a new point to an ordered configuration exactly when it consists of $2$ or more ordered points.
	\end{maintheorem}
	
	\begin{maintheorem}[\textbf{Unordered}]\label{theorem:nosection}
		In dimension $m \geq 2$ and with $n \geq 1$ points, the forgetful map $\bar{g}_{m,n}$ has a section if and only if $n=2$. That is, one \underline{cannot} continuously add a new point to an unordered configuration unless it consists of precisely $2$ unordered points.
	\end{maintheorem}
	
	As for the exceptional case of $n=2$,
	\begin{maintheorem}[\textbf{Uniqueness for} $n=2$]\label{thm:uniqueness}
		In every dimension $m\geq 1$, the midpoint construction of Example \ref{ex:case2} homotopically unique. That is, every section of $g_{m,2}$ or $\bar{g}_{m,2}$ is homotopic to the midpoint construction via a homotopy through sections.
	\end{maintheorem}
	
	Note that the case of points on a line segment ($m=1$) is excluded from Theorem \ref{theorem:nosection}. In this case, the unordered version coincides with the ordered one, as the points in an unordered configuration are nevertheless forced into a linear order. We may therefore add a point continuously to a configuration of $n$ points in $\DD^1$ as long as $n \neq 1$.
	
	Recasting Theorem \ref{theorem:nosection} as a direct generalization of Brouwer's fixed-point theorem one can say that, except for when $n=2$, every continuous map $f:\Conf_n(\DD) \rightarrow \DD$ has a `fixed point'. This is meant in the sense that there must exist some configuration $S=\{p_1,\ldots,p_n\}$ whose image under $f$ lies inside $S$.
	
	We remark that the negative results in Theorem \ref{theorem:nosection} contrasts with the analogous problem of introducing a new point to an unordered configuration of points in $\RR^m$ (or equivalently, on the \emph{open} ball). This latter problem always has a solution: add a point `at infinity', i.e., place it very far away from all the others  \cite{Chen}. Such a construction is of course not possible on the closed ball. However, since the configuration spaces of the open and closed balls are homotopy equivalent (see \S\ref{sec:unordered}), there is no purely homotopy theoretic obstruction to finding a section in the case of a closed ball. This means that a different approach is needed. Even more, many standard tools of algebraic topology fail in our context since the forgetful map $g_{m,n}$ is not a fibration: it has fibers of distinct homotopy types, depending on how many points lie on the boundary of the ball.
	
	\para{Outline} We treat the easy ordered version of the problem in \S\ref{sec:ordered} by briefly proving Theorem \ref{theorem:yessection}. The main argument of our proof of Theorem \ref{theorem:nosection} proceeds by contradiction. Any section of $\bar{g}_{m,n}$ induces an $\Sigma_n$-equivariant section of $g_{m,n}: \PConf_{n+1}(\DD) \to \PConf_{n}(\DD)$. Theorem \ref{theorem:nosection} is then proved by pulling back a cohomology class in two different ways and arriving at a contradiction. We conclude by proving Theorem \ref{thm:uniqueness} in \S\ref{sec:n=2}.

	\para{Acknowledgements} We thank Benson Farb and Dan Margalit for providing useful comments and for their continued encouragement and support. We also thank the anonymous referee for a number of suggestions that improved the paper. JL is supported by the NSF grant DGE-1650044.
	
	\section{Case of ordered configurations}
	\label{sec:ordered}
	
	In this section we prove Theorem \ref{theorem:yessection} by generalizing the midpoint construction of Example \ref{ex:case2}.
	
	\begin{proof}[Proof of Theorem \ref{theorem:yessection}]
		The case $n=1$ is covered by Example \ref{ex:case1}. Otherwise, fix $n\geq 2$. For any pair $(i,j)$ with {$1\leq i, j \leq n$ and $i\neq j$}, the line segment $[p_i,p_j]$ is contained within $\DD$. Let $v_{ij}:= p_j-p_i$ be the vector pointing from $p_i$ to $p_j$. Further, let $d_i:= \min{\{\lVert p_k-p_i \rVert  : k\neq i\}}$ be the minimal distance between $p_i$ and any other point in the configuration. One can then introduce a new point ``close" to $p_i$, lying at a distance ${d_i}/2$ from $p_i$ along the interval $[p_i,p_j]$: explicitly, $p_0 := p_i + \frac{d_i}{2\lVert v_{ij} \rVert}v_{ij}$. The added point $p_0$ is clearly distinct from all other points, and it depends on the configuration continuously since $d_i$ and $v_{ij}$ do so.
	\end{proof}
	
	We remark that the first author gave a prior classification of sections between configuration spaces (see the discussion preceding \cite[Theorem 1.1]{Chen}). Under this classification, the above construction of adding a point is of the type `add close to $p_i$'. Let us also call attention to the fact that the above construction relies on singling out two of the points, which cannot be done continuously in the unordered version of the problem when $n > 2$.
	
	\section{Case of unordered configurations}
	\label{sec:unordered}
	
	In this section we prove Theorem \ref{theorem:nosection}, which says that in dimension $m \geq 2$ there is no section of $\bar{g}_{m,n}$ except when $n=2$. We begin with a few preliminary observations.
	
	First, generators and relations for the integral cohomology ring of $\PConf_n(\RR^m)$ were produced by Arnol'd in dimension $m=2$ and by Cohen for general $m$ \cite{Arnold, cohen1988artin}. This cohomology ring is generated by classes 
	\[
	G_{ab}\in \Ho^{m-1}(\PConf_n(\RR^m);\ZZ)
	\]
	for distinct $1\leq a< b \leq n$, where the generator $G_{ab}$ measures the winding of point $a$ around point $b$.	The induced $\Sigma_n$-action is given by $\sigma^*(G_{ab})=G_{\sigma(a)\sigma(b)}$, where $G_{ba}=(-1)^m G_{ab}$.
	This description applies equally well to the closed ball as the following argument shows. Since $\RR^m$ is homeomorphic to the open unit ball $\BB$, their configuration spaces are also equivariantly homeomorphic. Then scaling by $0<t< 1$ gives a sequence of inclusions
	\[
	t\BB \subset t\DD \subset \BB \subset \DD
	\]
	with compositions isotopic to the identity. It follows that $\BB$ and $\DD$ have equivariantly homotopy equivalent configuration spaces, and in particular they have the same cohomology.
	
	Our second observation is that any section $\bar{s}$ of $\bar{g}_{m,n}$ for the unordered configuration spaces lifts to a $\Sigma_n$-equivariant section $s$ of $g_{m,n}$ for the ordered spaces. This follows from the lifting criterion for connected coverings.
	
	Third, we leverage the fact that an equivariant section $s$ of $g_{m,n}$ induces a solution to a related section problem where the configurations have the added restriction that the point $p_1$ is constrained to the boundary sphere. This solution gives us a map, whose pullback on cohomology we compute in two ways, leading to a contradiction. More precisely, let $\BB\subset \DD$ denote the interior of the closed ball and consider the subspace of $\PConf_{n}(\DD)$ in which only the $1$st point lies on the boundary sphere $S^{m-1}$:
	\[
	B^1_{n} := \{ (p_1,\ldots,p_n)\in \PConf_{n}(\DD) \mid p_1\in \partial \DD,\, (p_2,\ldots,p_n)\in \PConf_{n}(\BB) \} \cong S^{m-1}\times \PConf_{n-1}(\BB)
	\]
	Define $B^1_{n+1}\subseteq \PConf_{n+1}(\DD)$ similarly. Lastly, let $E^1_{n+1}$ denote the preimage $g_{m,n}^{-1}(B^1_n)$ and consider the inclusion $B^1_{n+1}\hookrightarrow E^1_{n+1}$. The difference between these two spaces is that in the larger space the additional point $p_0$ may lie on the boundary $\partial \DD$, while in the smaller space this is not allowed. Despite this apparent difference, we have the following.
	
	\begin{lem} \label{lem:B=E}
		The inclusion $B^1_{n+1}\hookrightarrow E^1_{n+1}$ is a homotopy equivalence.
	\end{lem}
	\begin{proof}
		Choosing an inward-pointing vector field that vanishes on the point $p_1\in \partial \DD$, one can push any other point into the interior of $\DD$. Explicitly, the vector field $-(p-p_1)$ on $\DD$ gives rise to a smooth vector field on $\PConf_{n+1}(\DD)$. Its flow produces an isotopy $\Phi^1_t$ such that for all $t>0$,
		\[
		\Phi^1_t(B^i_{n+1})\subset \Phi^1_t(E^i_{n+1}) \subset B^1_{n+1} \subset E^1_{n+1}
		\]
		thus establishing the homotopy equivalence.
	\end{proof}
	
	Consider the compatible projections onto the $1$st coordinate
	\[
	\xymatrix{
		B^1_{n+1} \ar[dr] \ar[d] _{g_{m,n}} & \\
		B^1_{n} \ar[r] & S^{m-1}.
	}
	\]
	Pulling back an orientation class $[S^{m-1}]\in \Ho^{m-1}(S^{m-1})$, we get a class $0\neq X_1\in \Ho^{m-1}(B^1_n)$ whose pullback will also be abusively denoted by $X_1\in \Ho^{m-1}(B^1_{n+1})$. These classes measure how many times the point $p_1$ wraps around the boundary sphere.
	
	\begin{lem} \label{lem:pullback}
		Under the inclusion $\iota: B^1_{n+1} \hookrightarrow \PConf_{n+1}(\DD)$ we have
		\[
		\iota^*(G_{1a}) = X_1 \text{ for all } 1<a
		\]
		and the class $G_{ab}$ for $1< a,b$ pulls back to $G_{ab}\in \Ho^{m-1}(\PConf_{n}(\BB))$. The same is true for $B^1_n\subseteq \PConf_{n}(\DD)$.
	\end{lem}
	
	Via the homotopy equivalence $B^1_{n+1}\simeq E^1_{n+1}$, we consider Lemma \ref{lem:pullback} as a statement about $E^1_{n+1}$ as well. In particular we shall keep the notation $X_1$ for the corresponding class in $\Ho^{m-1}(E^1_{n+1})$.
	
	\begin{proof}[Proof of Lemma \ref{lem:pullback}]
		These facts are geometrically obvious: the class $G_{ab}$ is pulled back from $S^{m-1}$ under the `Gauss map', sending a configuration to the direction vector from $p_a$ to $p_b$. When $1< a,b$, this Gauss map factors through the projection $B^1_{n+1} \to \PConf_{n}(\BB)$, as claimed.
		
		Otherwise, if $1< a$ then since $p_1$ lies on the boundary and $p_a$ is internal, the Gauss map is homotopic to a map in which $p_a$ is fixed at the origin. But when $p_a=0$ the Gauss map coincides with the projection which records only $p_1$.
	\end{proof}
	
	With these facts in hand, we can now complete our proof.
	\begin{proof}[Proof of Theorem \ref{theorem:nosection}]
		Let $m \geq 2$. If $n=1$, we have no section by Example \ref{ex:case1}. If $n=2$, we have a section by Example \ref{ex:case2}.  Otherwise, let $n \geq 3$. Assume that $\bar{s}$ is a section of $\bar{g}_{m,n}$ and let $s$ be its $\Sigma_n$-equivariant lift to a section of $g_{m,n}$. The assumption that $s$ is a section forces $s(B^1_n) \subseteq E^1_{n+1}$, thus it restricts to a section $s':B^1_n \to E^1_{n+1}$ of $g_{m,n}$. From this one observes that $(s')^* X_1=X_1$. Let us show that this leads to a contradiction.
		
		Since the classes $G_{ab}$ span $\Ho^{m-1}(\PConf_{n}(\DD);\ZZ)$, there is an expansion
		\[
		s^*(G_{01}) = \sum_{1<a\leq n} \lambda_{a}G_{1a} + \sum_{1< a< b \leq n}\delta_{ab} G_{ab}
		\]
		for some integer coefficients. Equivariance implies that a permutation $\sigma\in \Sigma_n$ fixing $1$ will preserve this expansion, and therefore $\lambda_a = \lambda_b$ for all $1<a<b\leq n$. Denote this constant value by $\lambda$. Similarly, it also follows that all $\delta_{ab}$ must be equal, say to the constant $\delta$. We thus get
		\begin{equation}\label{eq:expansion}
		s^*(G_{01}) = \lambda \sum_{1<a\leq n} G_{1a} + \delta \sum_{1<a<b\leq n} G_{ab}.
		\end{equation}	

		Next, we have a commutative diagram
		\[
		\xymatrix{
			E^{1}_{n+1} \ar[r]^{\iota\qquad} & \PConf_{n+1}(\DD) \\
			B^1_{n} \ar[u] ^{s'} \ar[r]^{\iota\qquad} & \PConf_{n}(\DD) \ar[u] ^s
		}
		\]
		through which the pullback of the class $G_{01}\in \Ho^{m-1}(\PConf_{n+1}(\DD))$ along the two different paths must agree. Pulling it back through the top-left corner,
		\[
		G_{01} \overset{\iota^*}\longmapsto X_1 \overset{(s')^*}\longmapsto X_1.
		\]
		But pulling back through the bottom right corner, one first applies $s^*$ for which we have the expansion \eqref{eq:expansion}. Since the restriction of $G_{ab}$ with $1<a,b$ to $B^1_n$ gives the class $G_{ab}$ again and each $G_{1a}$ restricts to $X_1$, we obtain the following:
		\[
		\iota^*s^*(G_{01})=\lambda(n-1)X_1+\delta \sum_{1<a<b\leq n} G_{ab}.
		 \]
		The above equation has to be equal to the pulling back from the top-left corner, which is $X_1$. But since the K\"{u}nneth formula for the product $B^1_n \cong S^{m-1}\times \PConf_{n-1}(\BB)$ implies that $X_i$ is linearly independent from the other classes $G_{ab}$, such an equality is possible only if $\delta=0$ and $\lambda(n-1)=1$. But $\lambda$ is an integer and $n > 2$ which gives a contradiction.
	\end{proof}
	
	\section{Uniqueness of the midpoint construction for $n=2$}
	\label{sec:n=2}
	
	We now show that the midpoint section from Example \ref{ex:case2} is unique up to a homotopy through sections.
	
	\begin{proof}[Proof of Theorem \ref{thm:uniqueness}]
		Let us denote the midpoint section by $M:\PConf_2(\DD)\to \PConf_{3}(\DD)$ and suppose that $s$ is another section of $g_{m,2}$, possibly $\Sigma_2$-equivariant. We construct a homotopy between $s$ and $M$ through sections, such that if $s$ was equivariant then so will be the homotopy.
		
		If for every configuration $s(p_1,p_2)=(p_0,p_1,p_2)$, the added point $p_0$ lies either between $p_1$ and $p_2$ or off of the line connecting them, then the straight-line homotopy $H_t:= (1-t)s+tM$ demonstrates the claim. With this, the uniqueness problem is reduced to finding a homotopy to a section possessing this property.
		
		The idea goes as follows. Given a configuration $(p_1,p_2)$ we let the points repel each other, moving outwards along the line connecting them until they hit the boundary. By applying the section $s$ to this motion, one gets a path of configurations of $3$ points, where at the end of the process we have $p_1$ and $p_2$ on the boundary and $p_0$ somewhere in $\DD$. Thus if $p_0$ is on the line containing $p_1$ and $p_2$, it must lie between them. On such a configuration one can perform the straight-line homotopy to the midpoint. However, to remain within sections of $g_{m,2}$, we must apply the aforementioned process while globally scaling $\DD$ down so that overall the points $p_1$ and $p_2$ do not move.
		
		Now more explicitly, for any configuration $c = (p_1,p_2)$, let $L^c$ be the straight line passing through it. Then $L^c$ intersects $\partial\DD$ at two distinct points $q_1$ and $q_2$, labeled so that $q_i$ is closer to $p_i$. Let $x^c$ be the unique point on the line $L^c$ from which the ratios of $\| q_i-x_c \|$ to $\| p_i-x_c \|$ are equal for $i=1,2$, and denote this common ratio by $r^c$. Then the isotopy 
		\[
		h_t^c:v\mapsto ((1-t)+r^c t)(v-x^c) +x^c
		\]
		is a scaling of $\RR^m$ out from $x^c$ at a linear rate, and $h_0^c = \operatorname{Id}$ while $h_1^c: p_i\mapsto q_i$ for $i=1,2$. Note that since $h^c_t$ is scaling out from a point inside $\DD$, the image of $\DD$ contains $\DD$ at every time $t$.
		
		Now, since $h^c_t$ and $(h^c_t)^{-1}$ are injective, they induce $\Sigma_n$-equivariant isotopies of the configuration spaces $(H^c_t)^{\pm 1}:\PConf_{n}(\RR^m)\to \PConf_{n}(\RR^m)$ by applying them to a configuration diagonally.  Lastly, it is clear that $H^c_t(p_1,p_2)$ is contained in $\DD$ at all times $0\leq t \leq 1$, and that everything in sight depends on $c$ continuously (even algebraically).
		
		A homotopy through sections is then given by
		\[
		s_t(c) := (H^c_t)^{-1}\circ s \circ H^c_t(c)
		\]
		This map acts by expanding the ball, using $s$ to add a point at every time $t$ and then contracting the ball back -- thus producing a path of configurations in which $(p_1,p_2)$ never moves. The section $s_1$ has the property that allows us to connect it to $M$ by a straight-line homotopy. 
		
		Lastly, if $s$ was $\Sigma_2$-equivariant, then $s_t$ is a composition of equivariant maps and is thus itself equivariant.\end{proof}

	\bibliography{disk}{}
	\bibliographystyle{alpha}

\end{document}